\documentclass[12pt]{amsart}
\usepackage{hyperref}
\usepackage{amsfonts}
\usepackage{bbm}
\usepackage{esint}
\usepackage{mathrsfs}
\usepackage{amssymb,amsmath,amsfonts,amsthm,color}
\usepackage{setspace}
\numberwithin{equation}{section}

\DeclareMathOperator{\supp}{supp}

\allowdisplaybreaks

\theoremstyle{plain}
\newtheorem{theorem}{Theorem}[section]

\newtheorem{lemma}[theorem]{Lemma}
\newtheorem{corollary}[theorem]{Corollary}

\theoremstyle{definition}
\newtheorem{definition}[theorem]{Definition}
\newtheorem{remark}[theorem]{Remark}

\allowdisplaybreaks

\newcommand{\e}{\varepsilon}

\newcommand{\R}{\mathbb{R}}

\def\R{\mathbb R}

\def\12{\frac{1}{2}}

\def\b1 {\dot{B}_1^{-\alpha,1}}

\begin{document}

\setcounter{tocdepth}{3} \allowdisplaybreaks

\title[Variable Hardy spaces in the Dunkl Setting ]
{Variable Hardy spaces in the Dunkl Setting
}

\author[J Tan]{Jian Tan}



\begin{abstract}
In this paper, we aim to define a new variable Hardy spaces in the Dunkl setting by the discrete Littlewood--Paley square functions. We also consider the atomic decomposition characterization for the new variable Hardy spaces with the help of the Littlewood--Paley theory and Hardy spaces on spaces of homogeneous type in the sense of Coifman and Weiss. As applications, we obtain the variable Hardy spaces boundedness for Dunkl--Calder\'on--Zygmund singular integrals involving by the Euclidean metric and the Dunkl ``metric'' which is associated with the finite reflection groups on the Euclidean space. 
\end{abstract}
\maketitle

\section{\bf{Introduction and statement of main results}}

Let $R$ be a normalized root system in $\mathbb R^n$ so that $\langle\alpha, \alpha\rangle=2$ for $\alpha\in R$ and $R^+$ be a fixed positive subsystem, where 
$\langle \cdot, \cdot\rangle$ 
stands for the usual scalar product on $\mathbb R^n$. 
We denote by $G$ the finite reflection group generated by the reflections $\sigma_\alpha$ for all $\alpha\in R$, 
where $\sigma_\alpha(x)=x-\langle x,\alpha\rangle\alpha$
for $x\in \mathbb R^n$.
A multiplicity function $\kappa(\alpha)\ge 0$ is invariant under the natural action of $G$ defined on $R$, 
which will be fixed throughout this paper. For $\xi\in R^n$,
the Dunkl operators $T_\xi$ are defined by
$$T_\xi f(x)=\partial_\xi f(x) + \sum\limits_{\alpha\in R}\frac{\kappa(\alpha)}{2}\langle\alpha, \xi\rangle\frac{f(x)-f(\sigma_\alpha(x))}{\langle \alpha, x\rangle}.$$
We denote that $T_j=T_{e_j}$,
where $e_1, \cdots, e_n$ are the standard unit vectors of $\mathbb R^n.$
The Dunkl measure $dw(x)$ is defined by
\begin{eqnarray*}
dw(x)=\prod_{\alpha\in R}|\langle\alpha,x\rangle|^{\kappa(\alpha)}dx=\prod_{\alpha\in R^+}|\langle\alpha,x\rangle|^{2\kappa(\alpha)}dx
\end{eqnarray*}
be the associated measure in $\mathbb{R}^{n}$, where $dx$ stands for the Lebesgue
measure in $\mathbb{R}^{n}$. 
\noindent We denote by $L^p(\R^n,dw)$, $p\in[1,\infty]$, the space of measurable functions on
$\mathbb R^n$ such that
$$
\|f\|_{L^p(\R^n,dw)}=\left(\int_{\mathbb R^N}|f(x)|^pdw(x)\right)^{1/p}<\infty,\;1\le p<\infty,
$$
$$
\|f\|_{L^\infty(\R^n,dw)}=\mbox{ess}\sup_{x\in \mathbb R^n}|f(x)|<\infty.
$$
For fixed $y\in R^n$, the Dunkl kernel $x\rightarrow E(x,y)$ is a unique analytic solution to the system
$$T_\xi f=\big<\xi,y\big>f,\quad f(0)=1$$
for all $\xi\in\mathbb R^n$.
In \cite{D2}, Dunkl introduced the Dunkl transform, which enjoys properties
similar to the classical Fourier transform. The transform is
defined by
$$ \widehat{f}(x)=c_h\int_{\mathbb{R}^{n}} E(x, -iy)f(y)h^2_{\kappa}(y)dy,$$
where the usual character $e^{-i\langle x,y\rangle}$ is replaced by
$E(x, -iy)=V_{\kappa}(e^{-i\langle \cdot,y\rangle})(x)$ for some
positive linear operator $V_{\kappa}$ and the weight functions
$h_{\kappa}$ are invariant under a finite reflection group $G$ on
$\mathbb{R}^{n}$. Particularly, the Dunkl transform also
satisfies the Plancherel identity, namely, $\|\widehat{f}\|_2=\|f\|_2$
and  if the parameter $\kappa=0,$ then $h_{\kappa}(x)=1$ and
$V_{\kappa}$ is the identity operator, thus the Dunkl transform reduces to the classical
Fourier transform automatically. In \cite{TX1}, the translation
operator related to Dunkl transform then is defined on the Dunkl
transform side by
$$ {\widehat{\tau_y f}}(x)=E(y, -ix) \hat{f}(x)$$
for all $x\in \mathbb{R}^{n}.$ When the function $f$ is in the Schwartz class, the above equality holds pointwise. As an operator on $L^2(\mathbb R^n,h^2_{\kappa}),$ $\tau_y$ is bounded. However, it is not at all clear whether the translation operator can be defined for $L^p$ functions with
$p\not=2.$ For $f, g\in L^2(\mathbb{R}^{n},h^2_{\kappa}),$ their convolution
can be defined in terms of the translation operator by
$$f{\ast}_{\kappa} g(x)=\int_{\mathbb{R}^{n}}f(y)\tau_{x}\widetilde g(y)h^2_{\kappa}(y)dy,$$
where $\widetilde g(y)=g(-y).$ 
However, in the Dunkl setting, there are the Dunkl transform,  Dunkl translations and Dunkl convolution conducted by finite reflection groups,
which is different from the classical one. 

Note that the Dunkl-Riesz transforms are defined by $$R_j(f)=-T_j(\triangle)^{-1/2}f=-c\int^\infty_0 T_je^{t\triangle}f\frac{dt}{\sqrt t}$$ and $e^{t\triangle}f(x)=\int_{\R^N} h_t(x,y)f(y)d\omega(y),$ where the integral converges in $L^2$ and $h_t(x,y)$ is the heat kernel.  
Very recently, C-Q. Tan et al. \cite{THHLL1} introduced the Dunkl--Calder\'on--Zygmund singular integral operators with the kernels involving both the Euclidean metric and the Dunkl metric deduced by the finite reflection groups. These new class of singular integrals covers the well-known Dunk--Riesz transforms and generalizes the classical Calder\'on--Zygmund singular integrals on spaces of homogeneous type. Furthermore, 
the  $T1$ theorem, Littlewood--Paley theory, Dunkl--Hardy spaces and the boundedness of such singular integrals on the Dunkl--Hardy spaces are also established. 

On the other hand, due to many applications to 
elasticity, fluid dynamics, calculus of variations, and differential equations with 
non-standard growth condition,
the variable exponent function spaces,
such as the variable Lebesgue spaces and the variable
Sobolev spaces, were studied by a substantial number
of researchers (see \cite{CF,DHHR2011} and references therein).
In $\mathbb R^n$,
for any variable exponent function $p(x)$ fulfilling
$$1\le p^-:=\mathrm{ess}\inf \{p(x):x\in\R^n\}\le p(x)\le p^+:=\mathrm{ess}\sup \{p(x):x\in\R^n\}<\infty,$$
the variable Lebesgue space $L^{p(\cdot)}(\mathbb R^n,dx)$ associated with Lebesgue measure $dx$ is defined to be the set of all measurable functions $f(x)$ defined on $\mathbb R^n$ such that
$$\int_{\mathbb R^n} |f(x)|^{p(x)}d x<\infty,$$
equipped with the Luxemburg--Nakano norm
$$\|f\|_{L^{p(\cdot)}(\mathbb R^n, dx)}
=\inf\left\{\lambda>0: \int_{\mathbb R^n} \left|\frac{f(x)}{\lambda}\right|^{p(x)}d x\le1 \right\}.$$
Obviously, when $p(x)=p$, the Lebesgue space with variable exponents $L^{p(\cdot)}(\mathbb R^n, dx)$ is reduced to the classical one $L^p(\mathbb R^n, dx)$.
As a special case of the theory of Nakano and Luxemberg,
we see that $L^{p(\cdot)}(\mathbb R^n, dx)$
is a quasi-normed space. Especially, when $p^-\geq1$, $L^{p(\cdot)}(\R^n, dx)$ is a Banach space.

In \cite{TT}, the first author and Tao have introduced the Dunkl--Lebesgue spaces with variable exponents and proved that the Dunkl--Calder\'on--Zygmund operator is bounded on these new Lebesgue spaces with variable exponents by using the pointwise estimates for Dunkl sharp function associated with Dunkl--Calder\'on--Zygmund operator. 
Thus, it is natural to ask the following question: {\bf Is the Dunkl--Calder\'on--Zygmund operator also bounded on the corresponding Hardy spaces with variable exponents?} 

The purpose of this paper is to give an {\bf affirmative answer} to this question.
To achieve it, we will introduced the new variable Dunkl--Hardy spaces via using the Littlewood--Paley theory.
Then we will give the atomic decomposition characterizations for these new Hardy spaces. Lastly, the boundedness of Dunkl--Calder\'on--Zygmund operator is obtained with the help of the atomic decomposition theory and the method of variable exponent analysis.

To describe our main results of  the paper, first we recall the weak-type wavelet decomposition on $L^2(\mathbb R^n)$ established in \cite{THHLL1}. 
Applying the following Calder\'on reproducing formula given in \cite{ADH},

\begin{equation}\label{ccrf}
f(x)=\int_0^\infty {\psi_{t}}\ast {q_{t}}\ast
f(x)\frac{dt}{t},
\end{equation}
where the integral converges in $L^2(\mathbb R^n)$ and ${q_t}f=t\frac{\partial}{\partial t}p_tf$ with $p_t$ is
the Dunkl--Poisson kernel, and ${\psi}\in C^\infty_0(B_d(0,1/4))$
is a radial function with $\int_{\mathbb R^nN}{\psi}(x)d\omega(x)=0$.
the authors in \cite{THHLL1} decompose $f$ by the following

$$f(x)=\int_0^\infty {\psi}_t\ast {q_t}\ast f(x)\frac{dt}{t}=T_M(f)(x) + R_1(f)(x) + R_M(f)(x),$$
where
$$T_M(f)(x)=-\ln r \cdot\sum\limits_{j=-\infty}^\infty\sum\limits_{Q\in {\mathcal Q}^j}w(Q) \psi_{j}(x,x_{Q})q_{j}\ast
f(x_{Q}),$$
$$R_1(f)(x)=-\sum\limits_{j=-\infty}^\infty\int_{r^{-j}}^{r^{-j+1}}
\Big[{\psi_t}\ast {q_t}\ast f(x)- \psi_{j}\ast q_{j}\ast
f(x)\Big]\frac{dt}{t}$$ and
$$R_M(f)(x)=-\ln r \cdot\sum\limits_{j=-\infty}^\infty\sum\limits_{Q\in {\mathcal Q}^j}\int_{Q}
\Big[\psi_{j}(x,y)q_{j}\ast f(y)-\psi_{j}(x,x_{Q})q_{j}\ast
f(x_{Q})\Big]d\omega(y),$$ where $\psi_j={\psi}_{r^j},
q_{j}={q}_{r^j}$ with $1<r\leqslant r_0$ for some fixed $r_0,
{\mathcal Q}^j$ is the collection of all ``$r$-dyadic cubes'' $Q$
with the side length $r^{-M-j}$ for $M$ is some fixed large integer,
and $x_{Q}$ is any fixed point in the cube $Q.$
By applying Coifman's decomposition of the identity on $L^2(\mathbb R^n)$
gives
$$I=T_M + R_1 + R_M.$$
The authors in \cite{THHLL1} showed that $R_1$ and $R_M$ are bounded on $L^p(\mathbb R^n, d\omega), 1<p<\infty,$ and moreover, $\|R_1+R_M\|_{p,p}< 1$ with  some a fixed $r_0>1$ and a fixed $M.$ This implies that $(T_M)^{-1},$ the inverse
of $T_M,$ exists and is bounded on $L^p(\mathbb R^n, d\omega), 1<p<\infty.$
Then the weak-type wavelet decomposition is given by the following theorem.

\begin{theorem}[\cite{THHLL1}]\label{th1.1}
    If $f\in L^2(\mathbb R^n,d\omega)$ then there exists a function $h\in L^2(\mathbb R^n,d\omega)$ such that $\|f\|_2\sim \|h\|_2$ and 
    \begin{align*}
    f(x)=\sum\limits_{j=-\infty}^\infty\sum\limits_{Q\in {\mathcal Q}^j}\omega(Q)
    \psi_{Q}(x,x_{Q})q_{Q}
    h(x_{Q}),
    \end{align*}
    where the series converges in $L^2(\mathbb R^n,d\omega)$ with $\psi_Q=\psi_{j}, q_{Q}=q_{j}$ when $Q\in {\mathcal Q}^j$ where $\{{\mathcal Q}^j\}$ is the collection of all ``r-dyadic cubes'' $Q$
     with the side length $r^{-M-j}$ for $M$ is some fixed large integer and $x_{Q}$ is any fixed point in the cube $Q.$
\end{theorem}

This weak-type wavelet decomposition leads to the following generalized discrete Littlewood--Paley function in  \cite{THHLL1}.
\begin{definition}\label{df1.2}
    For $f\in L^2(\mathbb R^n, d\omega), S(f),$ the \emph{discrete Littlewood--Paley square function} of $f,$ is defined by
    \begin{eqnarray}\label{square_function}
    S(f)(x):= \left\{
    \sum\limits_{j=-\infty}^\infty\sum\limits_{Q\in {\mathcal Q}^j}|q_{Q}f(x_{Q})|)^2\chi_{Q}
    (x) \right\}^{1/2},
    \end{eqnarray}
    where $\chi_{Q}(x)$ is the characteristic function of the cube $Q.$
\end{definition}

For our purpose, we first introduce the definition of the variable Dunkl--Hardy space.
In what follows, a measurable function $p(\cdot): \mathbb R\to (0,\infty)$ is said to be a variable exponent function.
Denote by $\mathcal{P}^0=\mathcal{P}^0(\mathbb R^n)$ the set of all variable exponents $p(\cdot)$ on $\mathbb R^n$ with $0< p^-\le p(x)\le  p^+<\infty$. 
Denote by $\mathcal{P}=\mathcal{P}(\R^n)$ the set of all variable exponents $p(\cdot)$ on $\R^n$ with $1< p^-\le p(x)\le  p^+<\infty$.
We assume throughout that the variable exponent $p(\cdot)$ is $G$-invariant, i.e., $p(\sigma x)=p(x)$ for all $\sigma\in G$ and $x\in\mathbb{R}^N$.
For convenience, we defer other technical definitions to Section 2.

\begin{definition}\label{df1.3}
    For $p(\cdot)\in \mathcal P^0$,  the variable Dunkl--Hardy space $H_D^{p(\cdot)}{(\mathbb R^n)}$ is defined as the completion of
    $$
    \left\{f\in L^2(\mathbb R^n): S(f)\in L^{p(\cdot)}(\mathbb R^n, d\omega)\right\}
    $$
     under the norm
$\|f\|_{H_D^{p(\cdot)}{(\mathbb R^n)}}:=\|S(f)\|_{L^{p(\cdot)}}(\mathbb R^n, d\omega)$.
\end{definition}

\begin{remark}
By applying the nearly identical argument to the proof of \cite[Theorem 1.13]{THHLL2} together with the boundedness of Dunkl--Calder\'on--Zygmund operator on variable Dunkl--Lebesgue space, 
we can obtain the weak-type discrete Calder\'on reproducing formula for $f\in L^2(\mathbb R^n, d\omega)$ with respect to $H_D^{p(\cdot)}{(\mathbb R^n)}$. That is,
    if $f\in L^2(\mathbb R^n, d\omega)\cap H_D^{p(\cdot)}(\mathbb R^n, d\omega)$ along with $p(\cdot)\in LH$ and $\frac{N}{N+1}<p^-\le p^+\le 1$, then
    $$f=\sum\limits_{j=-\infty}^\infty\sum\limits_{Q\in {\mathcal Q}^j}w(Q)
    \psi_{Q}(x,x_{Q})q_{Q}h(x_{Q})$$
    in the $L^2(\mathbb R^n, d\omega)$ norm and the $H_D^{p(\cdot)}{(\mathbb R^n)}$ quasi-norm,
    where $\|h\|_{L^2}\sim \|f\|_{L^2}$ and 
    $\|h\|_{H_D^{p(\cdot)}}\sim \|f\|_{H_D^{p(\cdot)}}.$
\end{remark}
If $p(\cdot)\in \mathcal P\cap LH$, then the following lemma yields that $H_D^{p(\cdot)}{(\mathbb R^n)} \cong  L^{p(\cdot)}(\mathbb R^n, d\omega)$
with equivalent norms. 
The Littlewood--Paley characterization for variable Dunkl--Lebesgue spaces $L^{p(\cdot)}(\mathbb R^n, d\omega)$ is given as follows.
\begin{theorem}\label{LP}
Let $p(\cdot)\in \mathcal P\cap LH$.  Then there exist two positive constants $C_1$ and $C_2$ such that for $L^{p(\cdot)}(\mathbb R^n, d\omega)$,
    $$C_1\|f\|_{L^{p(\cdot)}(\mathbb R^n, d\omega)}\leqslant \|S(f)\|_{L^{p(\cdot)}(\mathbb R^n, d\omega)}\leqslant C_2\|f\|_{L^{p(\cdot)}(\mathbb R^n, d\omega)}.$$
\end{theorem}

Atomic decomposition characterization plays a very important role in the real-variable theory of function spaces and
the boundedness of operators.
We first introduce the atom $a$ for $H_D^{p(\cdot)}(\mathbb R^n)$.
\begin{definition}
Suppose that $p(\cdot)\in LH,
0<p^-\le p^+<q<\infty$ and $q\ge 1$.
We say a function $a$ is a $(p(\cdot),q)-$atom of $H_D^{p(\cdot)}(\mathbb R^n)$,
if $a$ is supported in a cube $Q\subset\mathbb R^n$,
$$
\|a\|_{L^q(\mathbb R^n, d\omega)}\le \omega(Q)^{\frac{1}{q}}\|\chi_Q\|^{-1}_{L^{p(\cdot)}(\mathbb R^n, d\omega)},
$$
and
$$
\int_{\mathbb R^n}a(x)d\omega(x)=0.
$$
\end{definition}
The set of all such pairs $(a,Q)$ will be denoted by $ A(p(\cdot),q)$.
Hereafter, ${N}=n+\sum\limits_{\alpha\in R}\kappa(\alpha).$

Now we state the atomic decompositions for $H_D^{p(\cdot)}(\mathbb R^n)$.
\begin{theorem}\label{atom}
Suppose that $p(\cdot)\in LH, \frac{{N}}{{N}+1}<p^-\le p^+<\infty$. If $f\in L^2(\mathbb R^n, d\omega)\cap H_D^{p(\cdot)}(\mathbb R^n)$, there is a
sequence of $(p(\cdot),q)-$atoms $\{a_j\}$ and a sequence of non-negative scalars $\{\lambda_j\}$
with
\begin{equation*}
  \mathcal{A}(\{\lambda_j\}_{j=1}^\infty,\{Q_j\}_{j=1}^\infty)
  \leq C\|f\|_{H_D^{p(\cdot)}(\mathbb R^n)},
\end{equation*}
{\noindent}such that $f=\sum_j\lambda_ja_j$, where the series converges to
$f$ in both $L^2(\mathbb R^n, d\omega)$ and $H_D^{p(\cdot)}(\mathbb R^n)$ norms,
and
\begin{equation*}
  \mathcal{A}(\{\lambda_j\}_{j=1}^\infty,\{Q_j\}_{j=1}^\infty)
  =\left\|\sum_{j}
  \frac{\lambda_j\chi_{Q_j}}
  {\|\chi_{Q_j}\|_{L^{p(\cdot)}(\mathbb R^n, d\omega)}}
\right\|_{L^{p(\cdot)}(\mathbb R^n, d\omega)}
  <\infty.
\end{equation*}

Conversely, if $f$ has an atomic decomposition  $f=\sum_j\lambda_ja_j$ in $L^2(\mathbb R^n, d\omega)$,
then $f\in H_D^{p(\cdot)}(\mathbb R^n)$ and 
$$
\|f\|_{H_D^{p(\cdot)}(\mathbb R^n)}\le C\mathcal{A}(\{\lambda_j\}_{j=1}^\infty,\{Q_j\}_{j=1}^\infty).
$$
\end{theorem}

  Let ${\dot{C}}^\eta(\Bbb
  R^n)$ be the H\"older space of continuous functions $f$ with 
  $$\|f\|_{{\dot{C}}^\eta}:=\sup\limits_{x\ne y} \frac{|f(x)-f(y)|}{\|x-y\|^{\eta}}<\infty. $$
  We denote ${\dot{C}}^\eta_0(\Bbb R^n)$ by the H\"older space ${\dot{C}}^\eta(\Bbb R^n)$ with compact supports.   
 There are two important metrics in the Dunkl setting. The Euclidean metric is defined by 
$$\|x-y\|:=\left\{\sum\limits_{j=1}^n|x_j-y_j|^2\right\}^\frac{1}{2}$$ and the so-called Dunkl metric is given by
$$d(x,y)=\min\limits_{\sigma\in G}\|x-\sigma(y)\|,$$ the distance  between two G-orbits $\mathcal{O}(x)$ and $\mathcal{O}(y).$ Obviously, 
$$d(x,y)\leqslant \|x-y\|, d(x,y)=d(y,x),\quad \mbox{and}\quad d(x,y)\leqslant
d(x,z)+d(z,y)$$ for all $x,y,z\in \R^n.$ Moreover,
$\omega(B(x,r))\sim \omega(B(y,r))$ when $d(x,y)\sim r,$ and
$$\omega(B(x,r))\leqslant \omega(B_d(x,r))\leqslant
|G|\omega(B(x,r)),$$ 
where $B(x,r):=\{y\in
\mathbb{R}^{n}:\|x-y\|<r\}$ and 
$$B_d(x,r):=\{y\in
\mathbb{R}^{n}:d(x,y)<r\}.$$
Moreover, 
for all $x\in\mathbb{R}^n$ and $0<r_1<r_2$, there exists a positive constant $C$ such that
\[
C^{-1}\left(\frac{r_2}{r_1}\right)^n
\le
\frac{\omega(B(x,r_2))}{\omega(B(x,r_1))}
\le
C\left(\frac{r_2}{r_1}\right)^N.
\]

This implies that
$d\omega(x)$ satisfies the doubling and reverse
doubling properties, that is, there is a constant $C > 0$ such that
for all $ x\in \mathbb{R}^{n}, r>0,\lambda \geqslant 1$ and ${N}=n+\sum\limits_{\alpha\in R}\kappa(\alpha),$
$$C^{-1}\lambda^{n}\omega(B(x, r)) \leqslant \omega(B(x, \lambda
r))\leqslant C\lambda^{ N}\omega(B(x, r)).$$ 
It turns out that the Dunkl setting, $(\Bbb R^n,
  \|\cdot\|, dw),$ is space of homogeneous type in the sense of
  Coifman and Wiess with the measure $dw$ satisfing the
  doubling and the reverse doubling properties. 
Now we recall the following definition of Dunkl--Calder\'on--Zygmund operator which is introduced in \cite{THHLL1}.
\begin{definition}\label{df1.12}
        An operator $T: C_0^\eta(\mathbb{R}^n)\rightarrow(C_0^\eta(\mathbb{R}^n))'$ for some $\eta>0,$ is said to be a Dunkl--Calder\'on--Zygmund singular integral operator if $K(x,y),$ the kernel of $T$,
        satisfies the following conditions: there exists some $0<\varepsilon\leqslant 1$ such that
    \begin{equation}\label{si}
    |K(x,y)|\lesssim \frac1{\omega(B(x,d(x,y)))}\Big(\frac{d(x,y)}{\|x-y\|}\Big)^\varepsilon
    \end{equation}
    for all ${  d(x,y)\not=0};$
    \begin{equation}\label{smooth y3}
    |K(x,y)-K(x,y')|\lesssim \Big(\frac{\|y-y'\|}{\|x-y\|}\Big)^\varepsilon\frac{1}{\omega(B(x,d(x,y)))}
    \end{equation}
    for $\|y-y'\|\leqslant  d(x,y)/2;$
    \begin{equation}\label{smooth x3}
    |K(x',y)-K(x,y)|\lesssim \Big(\frac{\|x-x'\|}{\|x-y\|}\Big)^\varepsilon\frac1{\omega(B(x,d(x,y)))}
    \end{equation}
    for $\|x-x'\|\leqslant  d(x,y)/2.$

    Moreover,
    $$\langle T(f),g\rangle=\int_{\R^N}\int_{\R^n} K(x,y)f(x)g(y)d\omega(x)d\omega(y)$$
    for all $f$ and $g$ in $C_0^\eta(\mathbb{R}^n)$ with $\supp f\cap \supp g=\emptyset.$

    A Dunkl--Calder\'on--Zygmund singular integral operator is said to be the Dunkl--Calder\'on--Zygmund operator if it extends a bounded operator on $L^2(\mathbb R^n,\omega).$
\end{definition}

The boundedness of the Dunkl-Calder\'on-Zygmund operators on the
variable Dunkl--Lebesgue spaces is given in \cite{TT}. 
  \begin{theorem}\label{bdlp}
  	Suppose that $T$ is a Dunkl--Calder\'on--Zygmund operator and  $p(\cdot)\in  \mathcal P\cap LH$  with $\frac{{N}}{{N}+\epsilon}<p^-\le p^+<\infty$. Then $T$ extends to a bounded operator on the variable Dunkl--Lebesgue space $L^{p(\cdot)}(\mathbb R^n, dw)$. 
Moreover, there is a constant $C$ such that
  	$$\|T(f)\|_{L^{p(\cdot)}(\mathbb R^n, dw)} \leqslant C\|f\|_{L^{p(\cdot)}(\mathbb R^n, dw)}.$$ 	
  \end{theorem}
  
 Now we give the boundedness of the Dunkl--Calder\'on--Zygmund operators on the
variable Dunkl--Hardy spaces. 
   \begin{theorem}\label{bdhp}
  	Suppose that $T$ is a Dunkl--Calder\'on--Zygmund operator and  $p(\cdot)\in\mathcal P^0\cap LH$ with $\frac{{N}}{{N}+\epsilon}<p^-\le p^+<\infty$. 
Then $T$ extends to a bounded operator on the variable Dunkl--Hardy space $H_D^{p(\cdot)}(\mathbb R^n)$ to $L^{p(\cdot)}(\mathbb R^n,dw)$. 
Moreover, there is a constant $C$ such that
  	$$\|T(f)\|_{L^{p(\cdot)}(\mathbb R^n,dw)} \leqslant C\|f\|_{H_D^{p(\cdot)}(\mathbb R^n)}.$$ 	
  \end{theorem}
 
 We point out that the argument used in the proof of Theorem~\ref{bdhp}  only requires the regularity of $K$ in the second variable. In particular, no smoothness condition on the first variable $x$ is needed.
As a consequence of Theorem \ref{bdhp}, we also obtain the boundedness
of the Dunkl--Riesz transforms in the variable Dunkl-Hardy space
$H_D^{p(\cdot)}(\mathbb R^n)$.

\begin{corollary}\label{bdr}
    Let $p(\cdot)\in\mathcal P^0\cap LH$ with $\frac{{N}}{{N}+1}<p^-\le p^+<\infty$.
  Then the Dunkl--Riesz transforms $R_j, 1\leqslant j\leqslant N,$ are bounded from the variable Dunkl-Hardy spaces
    $H_D^{p(\cdot)}(\mathbb R^n)$ to $L^{p(\cdot)}(\mathbb R^n,dw)$.
\end{corollary}

The paper is organized as follows. In the next section we will give some preliminaries and notation. Then the Littlewood--Paley characterizations for 
variable Dunkl--Lebesgue spaces will be proved in Section 3. In this section, we will also present the atomic decomposition characterization for variable Dunkl--Hardy spaces.
The proofs of the boundedness of operators on the variable Dunkl--Hardy spaces will be included in Section 4.

\section{Preliminaries and notation}

In this section, we give an account on results from the Dunkl theory and the theory of the spaces of homogeneous type which will be relevant for the sequel.
Throughout this paper, $C$ or $c$ denotes a positive constant that is independent of the main parameters involved but may vary at each occurrence. To denote the dependence of the constants on some parameter $r$, we will write $C{(r)}$. We denote $f\leq Cg$ by $f\lesssim g$. If $f\lesssim g\lesssim f$, we write $f\sim g$.
In what follows, for any given $p^-\in(0, \infty)$, let $p_-:=\min\{p^-, 1\}.$

\subsection{Dunkl--Lebesgue spaces with variable exponents}
To define the variable function spaces in the Dunkl setting, let us first give the definitions of variable exponent functions in the Dunkl setting.
For any measurable functions $f(x)$ defined on $\R^n$ with respect to Dunkl measure $dw$, we also define the modular 
$$
\rho_{p(\cdot)}(f)=\int_{\R^n}|f(x)|^{p(x)}dw(x).
$$

\begin{definition}
For any $p(\cdot)\in \mathcal{P}$, the Dunkl--Lebesgue space with variable exponents $L^{p(\cdot)}(\R^n,dw)$
is defined to be the set of all measurable functions $f(x)$ defined on $\R^n$ with respect to Dunkl measure $dw$ such that the Luxemburg-Nakano norms
$$\|f\|_{L^{p(\cdot)}(\R^n, dw)}=\inf\left\{\lambda>0:\rho_{p(\cdot)}\bigg(\frac{1}{\lambda}f(x)\bigg)\le1 \right\}<\infty.$$
\end{definition}

We also recall some remarkable results on Lebesgue space with variable exponents in \cite{CC, HHP}, which can be applied to the Dunkl setting and will be used in our proofs of the main results.

By $M_Df$ we denote the Dunkl--Hardy--Littlewood maximal function of $f$, i.e.
$$
M_Df(x):=\sup_{t>0}\frac{1}{w(B(x,t))}\int_{B(x,t)}|f(y)|dw(y).
$$
We denote that $\widetilde M_Df(x):=\sum_{\sigma\in G}M_Df(\sigma(x))$.
By $\mathcal B$ we denote the set of bounded exponents $p(x)\in \mathcal P$ such that the function $M_D$ is bounded on $L^{p(\cdot)}(\R^n, dw)$.
An important subset of $\mathcal{B}$ is the well-known $LH$
condition.
A variable exponent function $p(\cdot) \in \mathcal{P}$ is said to be locally log-H\"{o}lder continuous in $\mathbb{R}^n$,
if there exists a constant $c_{\log }(p)>0$ such that for any $x, y \in {\R^n}$,
$$
|p(x)-p(y)| \leq \frac{c_{\log }(p)}{\log (\e+1/{\|x-y\|})},
$$
and is said to satisfy the log-H\"{o}lder decay condition if
there exist a $p_\infty\in\R$ and a constant $c_\infty(p)>0$ such that for any $x\in\R^n$,
$$
|p(x)-p_\infty| \leq \frac{c_\infty(p)}{\log (e+\|x\|)}.
$$
A variable exponent function $p(\cdot) \in \mathcal{P}$ is said to satisfy the globally log-H\"{o}lder continuous condition,
 denoted by $p(\cdot)\in LH$,
 if $p(\cdot)$ satisfies both the locally log-H\"{o}lder continuous condition and the log-H\"{o}lder decay condition.
We remark that, for the unbounded (quasi)metric measure space case, the above $LH$ definitions were first proved in \cite[Lemma 2.1]{AHH2015}.
Thus, as a consequence of \cite[Corollary 1.8]{AHH2015}, if the variable exponent function $p(\cdot)\in LH\cap \mathcal P$,
then the Dunkl--Hardy--Littlewood maximal function $M_Df$ of $f\in L^{p(\cdot)}(\R^n, dw)$
is bounded on $L^{p(\cdot)}(\R^n)$.
By using \cite[Theorem 1.1]{K}, if $p(\cdot)\in\mathcal B$, then the Dunkl--Hardy--Littlewood maximal function $M_Df$
is also bounded on $L^{p'(\cdot)}(\R^n)$, where $\frac{1}{p(x)}+\frac{1}{p'(x)}=1$ for any $x\in\mathbb R^n$, i. e. $p'(\cdot)\in \mathcal B$. 
The known Fefferman--Stein vector-valued inequality on spaces of homogeneous type $X=(\mathbb R^n, dw)$ in \cite[Theorem 2.7]{ZSY16} is also needed in our proofs.
\begin{lemma}\label{FS}
Let $p(\cdot)\in \mathcal P\cap LH$ and $u\in(1,\infty)$. Then for any measurable functions sequence $\{f_i\}_{i=1}^\infty\subset L^{p(\cdot)}(\mathbb R^n, dw)$,
\begin{equation*}
\left\|\left\{\sum_{i=1}^\infty[M_D(f_i)]^u\right\}^{\frac{1}{u}}\right\|_{L^{p(\cdot)}(\mathbb R^n, dw)}\le
C\left\|\left(\sum_{i=1}^\infty|f_i|^u\right)^{\frac{1}{u}}\right\|_{L^{p(\cdot)}(\mathbb R^n, dw)}.
\end{equation*}
\end{lemma}

\subsection{Littlewood--Paley theory and Hardy space on space of homogeneous type}
In this subsection, we recall the Littlewood--Paley theory and Hardy space on space of homogeneous type $(\mathbb R^n, \| \cdot\|, d\omega)$ in the sense of Coifman and Weiss. We begin with the following definition of the test functions in space of homogeneous type $(\mathbb R^n,\|\cdot\|,d\omega):$

    \begin{definition}\label{test} A function $f(x)$ defined on $\mathbb R^n$ is said to be a test function if there exits a constant $C$ such that for $0<\beta\leqslant 1, \gamma>0, r>0$ and $x_0\in \mathbb R^n,$
        \begin{enumerate}
            \item[(i)] $|f(x)|\leqslant \frac C{V(x, r+\|x-x_0\|)}\big(\frac{r}{r+\|x-x_0\|}\big)^\gamma;$
            \item[(ii)] $\displaystyle |f(x)-f(x')|\leqslant C\big(\frac{\|x-x'\|}{r+\|x-x_0\|}\big)^\beta\frac{1}{V(x,r+\|x-x_0\|)}\big(\frac{r}{r+\|x-x_0\|}\big)^\gamma, \\ \quad for \|x-x'\|\leqslant \frac{1}{2}(r+\|x-x_0\|);$
            \item[(iii)] $\displaystyle \int_{\R^N} f(x)d\omega(x)=0.$
        \end{enumerate}

        We denote such a test function by $f\in \mathcal M(\beta,\gamma,r,x_0)$ and $\|f\|_{\mathcal M(\beta,\gamma,r,x_0)}$, the norm in $\mathcal M(\beta,\gamma,r,x_0),$ is defined by the smallest $C$
        satisfying the above conditions (i) and (ii).
    \end{definition}

    Applying Coifman's approximation to the identity and decomposition of the identity operator together with the Calder\'on--Zygmund operator theory,
    the discrete Calder\'on reproducing formula in space of homogeneous type is given by the following

    \begin{theorem}\label{dCRh}
        Let $\{S_k\}_{k\in \Bbb Z}$ be a Coifman's approximation to the identity and set
        ${D}_k =: {S}_k - {S}_{k-1}$. Then there exist two families of operators
        $\{\widetilde{D}_k\}_{k\in \Bbb Z}$ and $\widetilde{\widetilde{D}}_k$ such that for any fixed $x_{Q}\in Q$ with $Q\in {\mathcal Q}^k$ with $k\in \Bbb Z$ and ${\mathcal Q}^k$ are "$r-$ dyadic cubes"
        with the side length $r^{-M-k},$
        $$f(x)=\sum\limits_{k=-\infty}^\infty\sum\limits_{Q\in {\mathcal Q}^k}\omega(Q){\widetilde D}_k(x,x_{Q})D_k(f)(x_{Q})=\sum\limits_{k=-\infty}^\infty\sum\limits_{Q\in {\mathcal Q}^k}\omega(Q)D_k(x,x_{Q})
        \widetilde{\widetilde{D}}_k(f)(x_{Q}),$$
        where the series converge in $L^p(\omega), 1<p<\infty, \mathcal M(\beta,\gamma,r,x_0),$ and in $(\mathcal M(\beta,\gamma,r,x_0))^\prime,$ the dual of in $\mathcal M(\beta,\gamma,r,x_0).$
        Moreover, the kernels of the operators $\widetilde{D}_k$
        satisfy the following conditions: for $0<\varepsilon\leqslant 1,$
        \begin{enumerate}
            \item[(i)] $|\widetilde{D}_k(x,y)|\leqslant C\frac1{V_k(x)+V_k(y)+V(x,y)}\Big(\frac{r^{-k}}{r^{-k}+\|x-y\|}\Big)^\varepsilon;$
            \item[(ii)] $|\widetilde{D}_k(x,y)-\widetilde{D}_k(x',y)|\leqslant C\Big(\frac{\|x-x'\|}{r^{-k}+\|x-x'\|}\Big)^\varepsilon\frac1{V_k(x)+V_k(y)+V(x,y)}\Big(\frac{r^{-k}}{r^{-k}+\|x-y\|}\Big)^\varepsilon$,
            \item[]for $\|x-x'\|\leqslant (r^{-k}+\|x-y\|)/2;$
            \item[(iii)] $\displaystyle \int_{\Bbb R^N} \widetilde{D}_k(x,y)d\omega(x)=0\qquad \text{for all}\ y\in \Bbb R^n;$
            \item[(iv)] $\displaystyle \int_{\Bbb R^N} \widetilde{D}_k(x,y)d\omega(y)=0\qquad \text{for all}\ x\in \Bbb R^n.$
            \end{enumerate}
And the kernels of $\widetilde{\widetilde{D}}_k(x,y)$ satisfy the
conditions \rm{(iii)--(iv)} and \rm{(i)--(ii)} with $x$ and $y$
interchanged.
    \end{theorem}
The discrete Calder\'on reproducing formula in Theorem \ref{dCRh} leads the
    following discrete square function on space of homogeneous type
    $(\Bbb R^n, \|\cdot\|, d\omega).$
    \begin{definition}\label{scw}
        Suppose that $f\in (\mathcal M(\beta,\gamma,r,x_0))^\prime.$ $S_{cw}(f),$ the Littlewood--Paley square function of $f$ on space of homogeneous type $(\mathbb R^n, \|\cdot\|, d\omega),$ is defined by
        $$S_{cw}(f)(x)=\bigg\{\sum\limits_{k=-\infty}^\infty\sum\limits_{Q\in {\mathcal Q}^k}
        |D_kf(x_{Q})|^2\chi_{Q}(x)\bigg\}^{1/2},$$
        where $D_k, {\mathcal Q}^k$ are same as given in Theorem \ref{dCRh}. See \cite{HMY} for more details.
    \end{definition}

\section {\bf{Littlewood--Paley characterization and atomic decomposition characterization}}
In this section, we will obtain the Littlewood--Paley characterizations for  variable Dunkl--Lebesgue spaces and present the atomic decomposition characterization for variable Dunkl--Hardy spaces with the help of the Littlewood--Paley theory and Hardy space on space of homogeneous type. First we give the proof of Theorem \ref{LP} via the theory of vector-valued Dunkl--Calder\'on--Zygmund operator.

\begin{proof}[\bf Proof of Theorem \ref{LP}]\quad
We write
    \begin{align*}
S(f)(x)&=\left(\sum\limits_{Q}\left|q_{Q}f(x_{Q})\right|^2\chi_Q(x)\right)^{\frac{1}{2}}
    =\left(\sum\limits_{Q}\left|\int_{\mathbb R^n}
    \chi_Q(x)q_{Q}(x_Q,y)f(y)d\omega(y)\right|^2\right)^{\frac{1}{2}}.
    \end{align*}
Then we can introduce the $H-$valued function $\{f_Q(x)\},$ where $Q$ are all $r$-dyadic cubes in $\mathbb R^n$. The norm of $\{f_Q(x)\}$ is defined by
    $$\|f_Q(x)\|_{H}:= \Big(\sum\limits_{Q}|f_Q(x)|^2\Big)^{\frac{1}{2}}.$$
Then let $T$ be an $H-$valued operator defined
    by $$T(f)(x)=\bigg\{\int_{\mathbb R^n} K(x,y)(f)(y)d\omega(y) \bigg\},$$
    where $K(x,y)=\{\chi_Q(x)q_{Q}(x_Q,y)\}$
From the proof of \cite[Theorem 1.11]{THHLL2} we know that $T$ is an $L^2-$bounded $H-$valued operator and the kernel $K(x,y)$ satisfies the following condition: for
    some $0<\varepsilon\leqslant 1,$
    \begin{equation*}\label{smooth y4}
    \|K(x,y)-K(x,y')\|_{H}\lesssim \Big(\frac{\|y-y'\|}{\|x-y\|}\Big)^\varepsilon\frac{1}{\omega(B(x,d(x,y)))} \qquad {\rm for}\ \|y-y'\|\leqslant  d(x,y)/2;
    \end{equation*}
Thus, repeating the same argument as in the proof of \cite[Theorem 1.5]{TT}, for any $p(\cdot)\in \mathcal P\cap LH$ we obtain that the Dunkl-Calder\'on-Zygmund oprator
$T$ is bounded from $L^{p(\cdot)}(\mathbb R^n,d\omega)$ to $L^{p(\cdot)}(H)$, 
which yields that the estimate $\|S(f)\|_{L^{p(\cdot)}(\mathbb R^n,d\omega)}\leqslant C\|f\|_{L^{p(\cdot)}(\mathbb R^n,d\omega)}$ for any $p(\cdot)\in \mathcal P\cap LH$.
On the other hand, by using the duality argument we get the estimate $\|f\|_{L^{p(\cdot)}(\mathbb R^n,d\omega)}\leqslant C\|S(f)\|_{L^{p(\cdot)}(\mathbb R^n,d\omega)}$ for any $p(\cdot)\in \mathcal P\cap LH$. 
Therefore, we have completed the proof of Theorem \ref{LP}.
\end{proof}

Before we establish the atomic decomposition characterization for variable Dunkl--Hardy space $H_D^{p(\cdot)}(\mathbb R^n,\omega)$, we first give
the relationship between the variable Dunkl--Hardy space $H_D^{p(\cdot)}(\mathbb R^n,d\omega)$ and the variable Hardy space $H_{cw}^{p(\cdot)}(\mathbb R^n, d\omega)$ on space of homogeneous type in the sense of Coifman and Weiss. Now we recall the definition of $H^{p(\cdot)}_{cw}(\mathbb R^n, d\omega)$ as follows.
\begin{definition}\label{hpcw}
For $p(\cdot)\in\mathcal P^0$,  the space $H^{p(\cdot)}_{cw}(\mathbb R^n,d\omega)$ is the collection of all distributions $f\in (\mathcal M(\beta,\gamma,r,x_0))^\prime$ such that
    $\|f\|_{H^{p(\cdot)}_{cw}(\mathbb R^n,d\omega)}=\|S_{cw}(f)\|_{L^{p(\cdot)}(\mathbb R^n,dw)}<\infty,$ where the square function $S_{cw}(f)$ is defined in the Definition \ref{scw}.
\end{definition}
The relationship between the Dunkl--Hardy space $H_D^{p(\cdot)}(\mathbb R^n)$ and Hardy space $H^{p(\cdot)}_{cw}(\mathbb R^n,d\omega)$ is given by the following

\begin{theorem}\label{relation}
Let $p(\cdot)\in\mathcal P^0\cap LH$ with $\frac{{ N}}{{ N}+1}<p^-\le p^+<\infty.$ Then for any $f\in H_D^{p(\cdot)}{(\mathbb R^n)}$ we have
$$\|f\|_{H_D^{p(\cdot)}{(\mathbb R^n)}}\sim \|f\|_{H^{p(\cdot)}_{cw}(\mathbb R^n,d\omega)}.$$
\end{theorem}

\begin{proof}
 First, we prove that for each $f\in L^2(\mathbb R^n,d\omega),  \|S(f)\|_{L^{p(\cdot)}(\mathbb R^n,d\omega)}\leqslant C \|S_{cw}(f)\|_{L^{p(\cdot)}(\mathbb R^n,d\omega)}$. 
By Theorem \ref{dCRh}, we have
    $f(x)=\sum\limits_{k=-\infty}^\infty\sum\limits_{Q\in {\mathcal Q}^k}\omega(Q)D_k(x,x_{Q})\widetilde{\widetilde{D}}_k(f)(x_{Q})$.
 Then we get that
    $$q_{Q'}f(x_{Q'})=\sum\limits_{k=-\infty}^\infty\sum\limits_{Q\in {\mathcal Q}^k}\omega(Q)q_{Q'}D_k(x_{Q'},x_{Q})\widetilde{\widetilde{D}}_k(f)(x_{Q}).$$
     This gives
    \begin{align*}
    S(f)(x)
    =\Big(\sum\limits_{k'=-\infty}^\infty\sum\limits_{Q'\in {\mathcal Q}^{k'}}\big\lvert\sum\limits_{k=-\infty}^\infty\sum\limits_{Q\in {\mathcal Q}^k}\omega(Q)q_{Q'}D_k(x_{Q'},x_{Q})
    \widetilde{\widetilde{D}}_k(f)(x_{Q})\big\rvert^2\chi_{Q'}(x)\Big)^{\frac{1}{2}}.
    \end{align*}

To prove it, we need the following claim.
        Let $k, k'\in\mathbb Z,r>1$ and any $x_{Q}\in Q$ and $x_{Q'}\in Q'$ with $Q\in {\mathcal Q}^k$ and $Q'\in {\mathcal Q}^{k'},$ respectively. Suppose that
        $$|S_{k',k}(x_{Q'},x_{Q})|\leqslant Cr^{-|k-k'|\varepsilon} \frac1{V(x_{Q'},x_{Q},r^{-k\vee -k'}+d(x_{Q'},x_{Q}))}
        \bigg(\frac{r^{-k \vee -k'}}{r^{-k \vee -k'} +d(x_{Q'},x_{Q})}\bigg)^{\varepsilon},$$
        for some $0<\varepsilon\leqslant 1.$ Then for $p(\cdot)\in\mathcal P^0\cap LH$ with $\frac{{ N}}{{N}+1}<p^-\le p^+<\infty$,
        \begin{align*}
        &\left\|\bigg(\sum\limits_{k'=-\infty}^\infty\sum\limits_{Q'\in {\mathcal Q}^{k'}}
        \Big|\sum\limits_{k=-\infty}^\infty\sum\limits_{Q\in {\mathcal Q}^k} \omega(Q)S_{k',k}(x_{Q'},x_Q)\lambda_Q\Big|^2\chi_{Q'}\bigg)^{\frac{1}{2}}\right\|_{L^{p(\cdot)}(\mathbb R^n,dw)} \\
        &\lesssim
        \left\| \bigg(\sum\limits_{k=-\infty}^\infty\sum\limits_{Q\in {\mathcal Q}^k} |\lambda_Q|^2\chi_{Q}\bigg)^{\frac{1}{2}}\right\|_{L^{p(\cdot)}(\mathbb R^n,dw)}.
        \end{align*}

From \cite[Lemma 2.13]{THHLL1} we know that for $\varepsilon\in (0,1)$
    \begin{align*}
    &|q_{Q'}D_k(x_{Q'},x_{Q})|   \lesssim r^{-|k-k'|\varepsilon} \frac1{V(x_{Q'},x_{Q},r^{-k\vee -k'}+d(x_{Q'},x_{Q})}
    \bigg(\frac{r^{-k \vee -k'}}{r^{-k \vee -k'} +d(x_{Q'},x_{Q})}\bigg)^{\varepsilon}.
    \end{align*}
    
By applying the above claim with $S_{k',k}(x_{Q'},x_{Q})=q_{Q'}D_k(x_{Q'},x_{Q})$ and $\lambda_Q= \widetilde{\widetilde{D}}_k(f)(x_{Q}),$ respectively, 
we have proved that $\|S(f)\|_{L^{p(\cdot)}(\mathbb R^n,d\omega)}\leqslant C \|S_{cw}(f)\|_{L^{p(\cdot)}(\mathbb R^n,d\omega)}$.
    
We now return to the proof of the claim. 
Repeating the similar argument in \cite[Section 4]{THHLL2}, we have
        \begin{align*}
        &\sum\limits_{k'=-\infty}^\infty\sum\limits_{Q'\in {\mathcal Q}^{k'}}
        \Big|\sum\limits_{k=-\infty}^\infty\sum\limits_{Q\in {\mathcal Q}^k} \omega(Q)S_{k',k}(x_{Q'},x_Q)\lambda_Q\Big|^2\chi_{Q'} \\
        &\lesssim \sum\limits_{\sigma\in G}\sum\limits_{k=-\infty}^\infty
        \bigg\{M\Big(\sum\limits_{Q\in {\mathcal Q}^k}|\lambda_Q|^\theta\chi_{Q}\Big)(\sigma(x))\bigg\}^{2/\theta},
        \end{align*}
        where $\sup\limits _{k\in \Bbb Z}\sum\limits_{k'=-\infty}^\infty r^{-|k-k'|\varepsilon}r^{[-k-(-k'\vee -k)]{ N}(1-\frac 1\theta)}<\infty$ for $\frac{N}{{ N}+\varepsilon}<\theta<p_-\leqslant 1.$
Finally, by Lemma \ref{FS}, we conclude that
        \begin{align*}
        &\left\| \bigg(\sum\limits_{k'=-\infty}^\infty\sum\limits_{Q'\in {\mathcal Q}^{k'}}
        \Big|\sum\limits_{k=-\infty}^\infty\sum\limits_{Q\in {\mathcal Q}^k} \omega(Q)S_{k,k'}(x_{Q'},x_Q)\lambda_Q\Big|^2\chi_{Q'}\bigg)^{\frac{1}{2}}\right\|_{L^{p(\cdot)}(\mathbb R^n,dw)}  \\
        \lesssim& \sum\limits_{\sigma\in G}
        \left\|\Big(\sum\limits_{k=-\infty}^\infty\sum\limits_{Q\in {\mathcal Q}^k}|\lambda_Q|^2\chi_Q(\sigma(x))\Big)^{\frac{1}{2}}\right\|_{L^{p(\cdot)}(\mathbb R^n,dw)}\lesssim
        \left\|\Big(\sum\limits_{k=-\infty}^\infty\sum\limits_{Q\in {\mathcal Q}^k}|\lambda_Q|^2\chi_Q(x)\Big)^{\frac{1}{2}}\right\|_{L^{p(\cdot)}(\mathbb R^n,dw)}.
        \end{align*}
 So the claim is proved.

Next we need to prove $\|S_{cw}(f)\|_{L^{p(\cdot)}(\mathbb R^n,dw)}\leqslant C\|S_{cw}(f)\|_{L^{p(\cdot)}(\mathbb R^n,dw)}$ for $f\in L^2(\mathbb R^n, dw)$.
To show it, we write
    \begin{align*}
    f(x)=T_M(f)(x)+R_1(f)(x)+R_M(x).
    \end{align*}
Then for $p_-\le 1$ we have
\begin{align*}
&\|S_{cw}(f)\|^{p_-}_{L^{p(\cdot)}(\mathbb R^n,dw)}\leqslant \|S_{cw}\big(T_M(f)(x)+R_1(f)(x)+R_M(x)\big)\|^{p_-}_{L^{p(\cdot)}(\mathbb R^n,dw)}\\
&\leqslant \|S_{cw}(T_Mf)\|^{p_-}_{L^{p(\cdot)}(\mathbb R^n,dw)} +\|S_{cw}(R_1(f))\|^{p_-}_{L^{p(\cdot)}(\mathbb R^n,dw)}
+\|S_{cw}(R_M(f))\|^{p_-}_{L^{p(\cdot)}(\mathbb R^n,dw)}.
\end{align*}

Now we show that
$$\|S_{cw}(R_1(f)\|_{L^{p(\cdot)}(\mathbb R^n,dw)}\leqslant C \|R_1\|_{dcz}  \|S_{cw}(f)\|_{L^{p(\cdot)}(\mathbb R^n,dw)}.$$
In fact,  by Theorem \ref{dCRh}, we obtain
    $$\|S_{cw}(R_1(f)\|_p=
    \bigg\|\bigg\{\sum\limits_{k=-\infty}^\infty\sum\limits_{Q\in {\mathcal Q}^k}
    |\sum\limits_{k'=-\infty}^\infty\sum\limits_{Q'\in {\mathcal Q}^{k'}}\omega(Q')D_kR_1D_{k'}(x_Q,x_{Q'})\widetilde{\widetilde{D}}_{k'}(f)(x_{Q'})|^2\chi_{Q'}(x)\bigg\}^{1/2}\bigg\|_p.$$ 
Applying \cite[Lemma 2.8]{THHLL1} to 
$D_kR_1D_{k'}(x_Q,x_{Q'})$ 
yields that for $\varepsilon\in (0,1)$ and any $x\in Q'$ with $Q'\in {\mathcal Q}^{k'},$
    \begin{align*}
    &|D_kR_1D_{k'}(x_Q,x_{Q'})| \\
    \lesssim& \|R_1\|_{dcz} r^{-|k-k'|\varepsilon} \frac1{V(x_{Q'},x_{Q},r^{-k\vee -k'}+d(x_{Q'},x_{Q}))}
    \bigg(\frac{r^{-k \vee -k'}}{r^{-k \vee -k'} +d(x_{Q'},x_{Q})}\bigg)^{\varepsilon}\chi_{Q'}(x).
    \end{align*}
Repeating the identical proof above together with the claim imply
    $$\|S_{cw}R_1(f)\|_{L^{p(\cdot)}(\mathbb R^n,dw)}
    \lesssim \|R_1\|_{dcz}\|S_{cw}(f)\|_{L^{p(\cdot)}(\mathbb R^n,dw)}.$$
    Similarly, we also can obtain 
 $$\|S_{cw}R_M(f)\|_{L^{p(\cdot)}(\mathbb R^n,dw)}
    \lesssim \|R_M\|_{dcz}\|S_{cw}(f)\|_{L^{p(\cdot)}(\mathbb R^n,dw)}.$$

Moreover, we get that
$$\|S_{cw}(R_1(f)\|_{L^{p(\cdot)}(\mathbb R^n,dw)}\leqslant 1/4 \|S_{cw}(f)\|_{L^{p(\cdot)}(\mathbb R^n,dw)}$$
and
$$\|S_{cw}(R_M(f)\|_{L^{p(\cdot)}(\mathbb R^n,dw)}\leqslant 1/4 \|S_{cw}(f)\|_{L^{p(\cdot)}(\mathbb R^n,dw)}.$$

Indeed, in \cite[Section 4]{THHLL2} we know that $\|R_1\|_{dcz}\leqslant C (r-1)$ for some $1<r\leqslant r_0$ with $r_0$ is closed 1
and $\|R_M\|_{dcz}\leqslant C r^{-M}.$ Therefore, for each $f\in L^2(\mathbb R^N, d\omega)$ we immediately obtain that $\|R_1\|_{dcz}\le 1/4$ 
and $\|R_M\|_{dcz}\le 1/4$.

Besides these estimates, we obtain that for $\frac { N}{{ N}+1}<p\leqslant 1,$
    $$\|S_{cw}(f)\|_{L^{p(\cdot)}(\mathbb R^n,dw)}
    \leqslant \|S_{cw}(T_Mf)\|_{L^{p(\cdot)}(\mathbb R^n,dw)} +\frac{1}{2}\|S_{cw}(f)\|_{L^{p(\cdot)}(\mathbb R^n,dw)}.$$
    
Set
    $$f_n(x)=\sum\limits_{|j|\le n}\sum\limits_{Q\in {\mathcal Q}^j\cap B(0,n)}w(Q)
    \psi_{Q}(x,x_{Q})q_{Q}h(x_{Q}).$$
    Note that $f_n$ coverges to $f$ in the $L^2(\mathbb R^N,\omega).$

Observe that
    $$\|S_{cw}(f_n)\|_{L^{p(\cdot)}(\mathbb R^n,dw)}\leqslant { 2} \|S_{cw}(T_Mf_n)\|_{L^{p(\cdot)}(\mathbb R^n,dw)}$$
    since $f_n$ has finite terms only.

Meanwhile, for $f\in L^2(\mathbb R^n, d\omega), \|S_{cw}(T_Mf)\|_{L^{p(\cdot)}(\mathbb R^n,dw)}\leqslant C \|S(f)\|_{L^{p(\cdot)}(\mathbb R^n,dw)}.$ 
Indeed, noting that
    $$T_M(f)(x)= -\ln r \cdot\sum\limits_{k'=-\infty}^\infty\sum\limits_{Q'\in {\mathcal Q}^{k'}}w(Q') \psi_{Q'}(x,x_{Q'})q_{Q'}f(x_{Q'}),$$
    and then
    \begin{align*}
    \bigg|D_kT_M(f)(x)\bigg|   \lesssim\sum\limits_{k'=-\infty}^\infty\sum\limits_{Q'\in {\mathcal Q}^{k'}}w(Q')|D_k \psi_{Q'}(x,x_{Q'})| |q_{Q'}f(x_{Q'})|.
    \end{align*}
By \cite[Lemma 2.8]{THHLL1} for $D_k \psi_{Q'}(x,x_{Q'})$  and then repeating the same proof above together with the claim we obtain
    \begin{align*}
    \|S_{cw}(T_Mf)\|_{L^{p(\cdot)}(\mathbb R^n,dw)}\lesssim \|S(f)\|_{L^{p(\cdot)}(\mathbb R^n,dw)}.
    \end{align*}
    Thus,
    $$\|S_{cw}(f_n)\|^{p_-}_{L^{p(\cdot)}(\mathbb R^n,dw)}\leqslant { 2} \|S_{cw}(T_Mf_n)\|^{p_-}_{L^{p(\cdot)}(\mathbb R^n,dw)}\leqslant C\|S(f_n)\|^{p_-}_{L^{p(\cdot)}(\mathbb R^n,dw)},$$
    which when $n$ tends to $\infty$ gives
    $$\|S_{cw}(f)\|^{p_-}_{L^{p(\cdot)}(\mathbb R^n,dw)}\leqslant C\|S(f)\|^{p_-}_{L^{p(\cdot)}(\mathbb R^n,dw)}.$$
Therefore, we have completed the proof of this theorem.
\end{proof}

From the proof of Theorem \ref{relation}, we know that the variable Dunkl--Hardy spaces is well defined. In the last of this section, we can show the proof of Theorem \ref{atom} by using the above theorem and the atomic decomposition for variable Hardy spaces  $H_{cw}^{p(\cdot)}(\mathbb R^n, d\omega)$.

\begin{proof}[\bf Proof of Theorem \ref{atom}]\quad
First, we know that the Dunkl setting, $(\mathbb R^n, \|\cdot\|, dw)$, is space of homogeneous type in the sense of Coifman and Wiess with the measure $dw$ satisfing the doubling and the reverse doubling properties.
Then, the atomic decomposition for variable Hardy spaces  $H_{cw}^{p(\cdot)}(\mathbb R^n, d\omega)$ on spaces of homogeneous type have been established in \cite[Theorem 4.3]{Tan2023} and \cite[Theorem 4.11]{YHYY}.
Thus, the atomic decomposition for $H_D^{p(\cdot)}{(\mathbb R^n)}$ follows from the atomic decomposition for  $H_{cw}^{p(\cdot)}{(\mathbb R^n)}$ and Theorem \ref{relation}. 
\end{proof}

\section{\bf{Boundedness of Dunkl--Calder\'on--Zygmund Operators}}
In this section, by applying atomic decomposition theory we will obtain the boundedness of Dunkl--Calder\'on--Zygmund operators on variable Dunkl--Hardy spaces.
Before we prove Theorem \ref{bdhp}, we first recall the following generalized Grafakos--Kalton lemma in \cite{Tan2023} which is
needed for our proofs.
Also see \cite[Section 4]{CMN} for more detail on the Euclidean case.
\begin{lemma}\cite[Lemma 4.4]{Tan2023}\label{GL}\quad
Given $p(\cdot)\in \mathcal P_0\cap LH$.
Fix $q>1$. Suppose that $0<p^+<q$, then
for given countable collections of cubes $\{Q_j\}_{j=1}^\infty$ and of nonnegative measurable functions $\{g_j\}_{j=1}^\infty$ such that
${\rm supp}(g_j)\subset Q_j,$
$$
\left\|\sum_{j=1}^\infty
g_j\right\|_{L^{p(\cdot)}(\mathbb R^n, d\omega)}
\le C\left\|\sum_{j=1}^\infty\left(\frac{1}{\omega(Q_j)}
\int_{Q_j}g_j^q\right)^{\frac{1}{q}}\chi_{Q_j}\right\|_{L^{p(\cdot)}(\mathbb R^n, d\omega)}.
$$
\end{lemma}

Now we give the proof of Theorem \ref{bdhp}.
\begin{proof}[{\bf The proof of Theorem \ref{bdhp}}]
Fix $1<q<\infty$. Observe that $L^q(\mathbb R^n,d\omega)\cap H_D^{p(\cdot)}(\mathbb R^n)$ is dense in $H_D^{p(\cdot)}(\mathbb R^n,d\omega).$ 
 By Theorem \ref{atom}, if $f\in L^q(\mathbb R^n, d\omega)\cap H_D^{p(\cdot)}(\mathbb R^n)$, 
 then for $p(\cdot)\in LH, \frac{{N}}{{N}+1}<p^-\le p^+<\infty$, there is a
sequence of $(p(\cdot),q)-$atoms $\{a_j\}$ and a sequence of non-negative scalars $\{\lambda_j\}$
with
\begin{equation*}
  \mathcal{A}(\{\lambda_j\}_{j=1}^\infty,\{Q_j\}_{j=1}^\infty)
  \leq C\|f\|_{H_D^{p(\cdot)}(\mathbb R^n)},
\end{equation*}
{\noindent}such that $f=\sum_j\lambda_ja_j$, where
\begin{equation*}
  \mathcal{A}(\{\lambda_j\}_{j=1}^\infty,\{Q_j\}_{j=1}^\infty)
  =\left\|\sum_{j}
  \frac{\lambda_j\chi_{Q_j}}
  {\|\chi_{Q_j}\|_{L^{p(\cdot)}}}
\right\|_{L^{p(\cdot)}(\mathbb R^n, d\omega)}
  <\infty.
\end{equation*}
and the series converges to
$f$ in both $L^2(\mathbb R^n, d\omega)$ and $H_D^{p(\cdot)}(\mathbb R^n)$ norms.
It follows that
$$
T(f)(x)=\sum_j\lambda_jT(a_j)(x),
$$
in the sense of $L^q(\mathbb R^n, d\omega)$ and $H_D^{p(\cdot)}(\mathbb R^n)$ norms.
Let $\mbox{supp}\,a_j(x)\subseteq Q_j$ and 
$$B^j_d=\{x: d(x,x_{Q_j})\leqslant 4{\sqrt n}\l(Q_j) \},$$ 
where $x_{Q_j}$ is the center of $Q_j$ and $\ell(Q_j)$ is the side length of $Q_j.$ 
Write
    $$\left\|\sum_j\lambda_jT(a_j)\right\|_{L^{p(\cdot)}(\mathbb R^n, d\omega)}
    =\left\|\sum_j\lambda_jT(a_j)\chi_{B_d^j}\right\|_{L^{p(\cdot)}(\mathbb R^n, d\omega)}
    +\left\|\sum_j\lambda_jT(a_j)\chi_{{(B_d^j)}^c}\right\|_{L^{p(\cdot)}(\mathbb R^n, d\omega)}.$$
Now we estimate the first term. From Theorem \ref{bdlp},
we conclude that $$
\|T(a_j)\|_{L^q(\mathbb R^n, d\omega)}
\le \|a_j\|_{L^q(\mathbb R^n, d\omega)}
\le \omega({Q_j})^{\frac{1}{q}}\|\chi_{Q_j}\|^{-1}_{L^{p(\cdot)}(\mathbb R^n, d\omega)}.
$$
Applying Lemma \ref{GL}, we get that
\begin{align*}
\left\|\sum_j\lambda_jT(a_j)\chi_{B_d^j}\right\|_{L^{p(\cdot)}(\mathbb R^n, d\omega)}
\le C\left\|\sum_j\frac{\lambda_j\chi_{B_d^j}}{\|\chi_{Q_j}\|_{L^{p(\cdot)}}}\right\|_{L^{p(\cdot)(\mathbb R^n, d\omega)}(\mathbb R^n, d\omega)}
\le C\mathcal{A}(\{\lambda_j\}_{j=1}^\infty,\{Q_j\}_{j=1}^\infty),
\end{align*}    
where the last inequality comes from the relationship between $B_d$ and $B$ and the Fefferman--Stein vector valued
maximal inequality in Lemma \ref{FS}.

To estimate the second term, observe that if $x\in {(B_d^j)}^c$ and $y\in Q_j,$ then $\|y-x_{Q_j}\|\leqslant \frac{1}{2}d(x,x_{Q_j})$. Applying the cancellation condition of $a$ and the smoothness condition of the kernel $K(x,y)$ yield that
\begin{align*}
 |T(a_j)(x)|&
 =\bigg|\int_{Q_j} K(x,y)a_j(y)d\omega(y)\bigg|=\bigg|\int_{Q_j} [K(x,y)-K(x,x_{Q_j})]a_j(y)d\omega(y)\bigg|\\
 \leqslant& C\Big(\frac{\l(Q_j)}{\|x-x_{Q_j}\|}\Big)^\varepsilon \frac{1}{\omega(B(x, d(x,x_{Q_j})))}\|a_j\|_{L^1(\mathbb R^n, d\omega)}\\
 \leqslant& C\Big(\frac{\l(Q_j)}{\|x-x_{Q_j}\|}\Big)^\varepsilon \frac{\omega{(Q_j)}}{\omega(B(x, d(x,x_{Q_j})))}
 \frac{1}{\|\chi_{Q_j}\|_{L^{p(\cdot)}(\mathbb R^n, d\omega)}}.
\end{align*}

We claim that
\begin{align}\label{claim22}
\Big(\frac{\l(Q_j)}{\|x-x_{Q_j}\|}\Big)^\varepsilon \frac{\omega{(Q_j)}}{\omega(B(x, d(x,x_{Q_j})))}
\le C\left(\widetilde{M}_D\chi_{{Q_j}}(x)\right)^{\frac{N+\varepsilon}{N}}.
\end{align}

Write $\tau=\frac{N+\varepsilon}{N}$. Since $p^->\frac{N}{N+\varepsilon}$, then $\tau p^->1$. 
Thus, by using Lemma \ref{FS} we have
\begin{align*}
&\left\|\sum_j\lambda_jT(a_j)\chi_{{(B_d^j)}^c}\right\|_{L^{p(\cdot)}(\mathbb R^n, d\omega)}\\
&\le C\left\|\sum_j\lambda_j\Big(\frac{\l(Q_j)}{\|\cdot-x_{Q_j}\|}\Big)^\varepsilon \frac{\omega{(Q_j)}}{\omega(B(\cdot, d(\cdot,x_{Q_j})))}
 \frac{1}{\|\chi_{Q_j}\|_{L^{p(\cdot)}(\mathbb R^n, d\omega)}}
 \chi_{{(B_d^j)}^c}\right\|_{L^{p(\cdot)}(\mathbb R^n, d\omega)}\\
&\le C\left\|\sum_j\frac{\lambda_j \widetilde M_D^{\tau}\chi_{{Q_j}}}{\|\chi_{Q_j}\|_{L^{p(\cdot)}(\mathbb R^n, d\omega)}}
\right\|_{L^{p(\cdot)}(\mathbb R^n, d\omega)}\\
&= C\left\|\left\{\sum_j\frac{\lambda_j \widetilde M_D^{\tau}\chi_{{Q_j}}}{\|\chi_{Q_j}\|_{L^{p(\cdot)}(\mathbb R^n, d\omega)}}
\right\}^{\frac{1}{\tau}}\right\|^\tau_{L^{p(\cdot)\tau}(\mathbb R^n, d\omega)}\\
&\le C\left\|\sum_j\frac{\lambda_j\chi_{{Q_j}}}{\|\chi_{Q_j}\|_{L^{p(\cdot)}(\mathbb R^n, d\omega)}}
\right\|_{L^{p(\cdot)}(\mathbb R^n, d\omega)}\\
&=\mathcal{A}(\{\lambda_j\}_{j=1}^\infty,\{Q_j\}_{j=1}^\infty).
\end{align*}

Finally, we only need to prove Claim \ref{claim22}.
Indeed, let
\[
R=d(x,x_{Q_j}).
\]
Since $x\in (B_d^j)^c$ and
\[
B_d^j=\{x:d(x,x_{Q_j})\le 4\sqrt n\,\l(Q_j)\},
\]
we have
\[
R\gtrsim \l(Q_j).
\]

By the standard estimate for the Dunkl Hardy--Littlewood maximal operator,
\[
\frac{\omega(Q_j)}
{\omega(B(x,R))}
\lesssim \widetilde M_D(\chi_{Q_j})(x),
\qquad x\in (B_d^j)^c
.
\]
Hence,
\[
\Big(
\frac{\omega(Q_j)}
{\omega(B(x,R))}
\Big)^{1+\frac{\varepsilon}{N}}
\lesssim \Big(\widetilde M_D(\chi_{Q_j})(x)\Big)^{\frac{N+\varepsilon}{N}}.
\]

On the other hand, 
from the fact that $\omega(B(x,r))\sim \omega(B(y,r))$ when $d(x,y)\sim r,$ 
 we know that $\omega(B(x,R))\sim \omega(B(x_{Q_j},R))$ as $R=d(x,x_{Q_j})$.

Moreover, by using the fact that,
for all $x\in\mathbb{R}^n$ and $0<r_1<r_2$, 
\[
\frac{\omega(B(x,r_2))}{\omega(B(x,r_1))}
\lesssim
\left(\frac{r_2}{r_1}\right)^N,
\]
and that $R>l(Q_j)>0$,
then we obtain
\[
\frac{\omega(B(x_{Q_j},R))}
{\omega(B(x_{Q_j},l(Q_j))}
\lesssim
\Big(\frac{R}{l(Q_j)}\Big)^N.
\]

Therefore,
\[
\Big(
\frac{\omega(Q_j)}
{\omega(B(x,R))}
\Big)^{\varepsilon/N}
\sim
\Big(
\frac{\omega(B(x_{Q_j},l(Q_j)))}
{\omega(B(x_{Q_j},R))}
\Big)^{\varepsilon/N}
\gtrsim
\Big(
\frac{\l(Q_j)}R
\Big)^\varepsilon.
\]

Consequently,
\[
\Big(
\frac{\l(Q_j)}R
\Big)^\varepsilon
\frac{\omega(Q_j)}
{\omega(x,R)}
\lesssim
\Big(
\frac{\omega(Q_j)}
{\omega(B(x,R))}
\Big)^{1+\varepsilon/N},
\]
which yields \eqref{claim22}.
Therefore, we complete the proof of Theorem \ref{bdhp}.  

 \end{proof}

{\bf Acknowledgments.}




The project is sponsored by
the National Natural Science Foundation of China(Grant No. 11901309), 
and the Open Project
Program of Key Laboratory of Mathematics and Complex System (Grant No. K202502), Beijing Normal University.
The main part of this work was completed during the author's visit to Auburn University. The author would like to thank Auburn University for its hospitality.

\bigskip
\bigskip

\medskip
\noindent Jian Tan\\
\noindent School of Science,\\
Nanjing University of Posts and Telecommunications,\\
Nanjing 210023, People's Republic of China\\

\noindent {\it E-mail address}: \texttt{tj@njupt.edu.cn}(J. Tan)\\


\bigskip

\begin{thebibliography}{99}
\bibitem{AHH2015}
T. Adamowicz, P. Harjulehto, P. H\"{a}st\"{o},
 Maximal operator in variable exponent Lebesgue spaces on unbounded quasimetric measure spaces.
Math. Scand. {\bf116} (2015), no. 1, 5--22.


\bibitem{ADH} J.-Ph. Anker, J. Dziuba\'nski and A. Hejna, Harmonic functions, conjugate harmonic functions and the Hardy $H^1$ in rational Dunkel setting,
J. Fourier Anal. Appl., {\bf25} (2019), no. 5, 2356--2418.


\bibitem{CC} D. Cruz-Uribe and J. Cummings, Weighted norm inequalities for the maximal operator on  $L^{p(\cdot)}$  
over spaces of homogeneous type.
Ann. Fenn. Math. {\bf 47} (2022), no. 1, 457--488.

\bibitem{CF} D. Cruz-Uribe and A. Fiorenza, \textit{Variable Lebesgue spaces:
Foundations and Harmonic Analysis}, Birkh\"{a}user, Basel, 2013.

\bibitem{CMN}  D. Cruz-Uribe, K. Moen and H. V. Nguyen, 
{\it A new approach to norm inequalities on weighted and variable Hardy spaces,} {Ann. Acad. Sci. Fenn. Math.,} 
{\bf 45} (2020), 175--198. 


\bibitem{DHHR2011}
L. Diening, P. Harjulehto, P. H\"{a}st\"{o}, M. R$\mathrm{\mathring{u}}$\v{z}i\v{c}ka,
 Lebesgue and Sobolev spaces with variable exponents.
 Lecture Notes in Mathematics, 2017. Springer, Heidelberg, 2011. x+509 pp. 
 

\bibitem{D2} C.F. Dunkl, Differential-difference operators associated to reflection groups,  Trans. Amer. Math., {\bf 311} (1989), no. 1,
167--183.


\bibitem{HMY} Y. Han, D. M\"{u}ller and D. Yang, A theory
    of Besov and Triebel-Lizorkin spaces on metric measure
    spaces modeled on Carnot--Carath\'eodory spaces, Abstr.
Appl. Anal., Vol.~2008, Article ID 893409. 250~pages.

\bibitem{HHP}  P. Harjulehto, P. H\"ast\"o and M. Pere, Variable exponent Lebesgue spaces on metric spaces: the Hardy--Littlewood maximal operator. Real Anal. Exchange {\bf30} (2004), no. 1, 87--103.


\bibitem{K} A. Y. Karlovich, Hardy--Littlewood maximal operator on reflexive variable Lebesgue spaces over spaces of homogeneous type. Studia Math. {\bf 254} (2020), 149--178.


\bibitem{THHLL1} C. Tan, Ya Han, Yo. Han, M.-Y. Lee, and J. Li, Criterion of the $L^2$ boundedness in Dunkl Setting and applications, Ann. Sc. Norm. Super. Pisa Cl. Sci. (5) {\bf 26} (2025), no. 1, 259--299.

\bibitem{THHLL2} C. Tan, Ya Han, Yo. Han, M.-Y. Lee, and J. Li, Singular integral operators, $T1$ theorem, Littlewood--Paley theory and Hardy spaces in Dunkl Setting. arXiv preprint arXiv:2204.01886, (2022).

\bibitem{Tan2023} J. Tan, A new approach for Hardy spaces with variable exponents on spaces of homogeneous type,
Filomat {\bf37} (2023), no. 23, 7719--7739.

\bibitem{TT} J. Tan and X. Tao, Dunkl--Calder\'on--Zygmund operators on Dunkl--Lebesgue spaces with variable exponents, Acta Math. Sinica (Chinese Ser.),
{\bf 68} (2025), no. 6.  923--936.

\bibitem{TX1} S. Thangvelu and Y. Xu, Convolution operator and maximal function for the Dunkl transform, J. Anal. Math., {\bf97} (2005), 25--55.

\bibitem{YHYY} X. Yan, Z. He, D. Yang and W. Yuan, Hardy spaces associated with ball quasi-Banach function spaces on spaces of homogeneous type: Littlewood--Paley characterizations with applications to boundedness of Calder\'on--Zygmund operators. Acta Math. Sin. (Engl. Ser.) {\bf 38} (2022), no.7, 1133--1184.

\bibitem{ZSY16} C-Q. Zhuo, Y. Sawano and D-C. Yang,
{Hardy spaces with variable exponents on RD-spaces and applications,} Dissertationes Math.  {\bf 520} (2016), 1--74.


\end{thebibliography}
\end{document}